\documentclass{amsart}
\usepackage{amsopn,amstext,amsbsy,amsmath,amsthm,MnSymbol}
\usepackage{geometry}
\usepackage[latin9]{inputenc}
\theoremstyle{plain}
\newtheorem{theorem}                 {\bf Theorem}        [section]
\newtheorem{proposition}  [theorem]  {\bf Proposition}

\newtheorem{lemma}        [theorem]  {\bf Lemma}

\theoremstyle{definition}

\newtheorem{definition}   [theorem]  {\bf Definition}

\allowdisplaybreaks

\numberwithin{equation}{section}

\begin{document}

\def\nab#1#2#3{\nabla^{\hbox{$\scriptstyle{#1}$}}_{\hbox{$\scriptstyle{#2}$}}{\hbox{$#3$}}}
\def\nabl#1#2{\nabla_{\hbox{$\scriptstyle{#1}$}}{\hbox{$#2$}}}
\def\tnabl#1#2{\hat{\nabla}_{\hbox{$\scriptstyle{#1}$}}{\hbox{$#2$}}}
\def\nnab#1{\nabla_{\hbox{$\scriptstyle{#1}$}}}
\def\tnab#1#2{\widetilde{\nabla}_{\hbox{$\scriptstyle{#1}$}}\hbox{$#2$}}
\def\rn{\mathbb{R}}
\def\cn{\mathbb{C}}
\def\zn{\mathbb{Z}}
\def\rk{\mathbb{K}}
\def\ci{\mathcal{I}}
\def\cL{\mathcal{L}}
\def\v{\mathfrak{v}}
\def\b{\mathfrak{b}}
\def\fr{\mathfrak{r}}
\def\z{\mathfrak{z}}
\def\g{\mathfrak{g}}
\def\gl{\mathfrak{gl}}
\def\sl{\mathfrak{sl}}
\def\k{\mathfrak{k}}
\def\p{\mathfrak{p}}
\def\h{\mathfrak{h}}
\def\s{\mathfrak{s}}
\def\so{\mathfrak{so}}
\def\su{\mathfrak{su}}
\def\n{\mathfrak{n}}
\def\m{\mathfrak{m}}
\def\a{\mathfrak{a}}
\def\V{\mathcal{V}}
\def\H{\mathcal{H}}
\def\E{\mathcal{E}}
\def\L{\mathfrak{L}}
\def\B{\mathcal{B}}
\def\Hi{\mathfrak{H}}
\def\U{\mathfrak{U}}
\def\F{\mathcal{F}}
\def\tr{\textrm{\upshape{trace}}}
\def\sp{\textrm{\upshape{span}}}
\def\grad{\textrm{\upshape{grad}}}
\def\div{\textrm{\upshape{div}}}
\def\ad{\textrm{\upshape{ad}}}
\def\Ad{\textrm{\upshape{Ad}}}
\def\Aut{\textrm{\upshape{Aut}}}
\def\Inn{\textrm{\upshape{Inn}}}
\def\Int{\textrm{\upshape{Int}}}
\def\Iso{\textrm{\upshape{Iso}}}
\def\Ker{\textrm{\upshape{Ker}}}
\def\Der{\textrm{\upshape{Der}}}
\def\dim{\textrm{\upshape{dim}}}
\def\Re{\textrm{\upshape{Re}}}
\def\id{\textrm{\upshape{id}}}
\def\Im{\textrm{\upshape{Im}}}
\def\End{\textrm{\upshape{End}}}
\def\rank{\textrm{\upshape{rank}}}
\def\pd#1{\frac{\partial}{\partial #1}}
\def\dd#1{\frac{\textrm{\upshape{d}}}{\textrm{\upshape{d}} #1}}
\def\SPE#1#2{\langle #1,#2\rangle}
\def\SPC#1#2{\left<#1,#2\right>}
\def\SPH#1#2{\left(#1,#2\right)}
\def\SPW#1#2{\llangle#1,#2\rrangle}
\def\SPS#1#2{\langlebar#1,#2\ranglebar}
\def\d{\textrm{\upshape{d}}}
\def\pro{\textsc{Proof}}

\title{Harmonic morphisms and eigenfamilies on the\\
exceptional Lie group $G_{2}$}

\date{}

\author{Jonas Nordström}




\address
{Department of Mathematics, Faculty of Science, Lund University,
Box 118, S-221 00 Lund, Sweden}
\email{Jonas.Nordstrom@math.lu.se}

\begin{abstract}

We construct harmonic morphisms on the compact simple Lie group $G_{2}$. The construction
uses eigenfamilies in a representation theoretic scheme.
\end{abstract}

\maketitle

\Large
\section{Introduction}

A harmonic morphism is a map between two Riemannian manifolds
that pulls back local real-valued harmonic functions to local real-valued harmonic functions.
The simplest examples of harmonic morphisms are constant maps,
real-valued harmonic functions and isometries.
A characterization of harmonic morphisms was given by Fuglede and Ishihara,
they show in \cite{Fuglede} and \cite{Ishihara}, respectively, that the harmonic
morphisms are exactly the harmonic horizontally weakly conformal maps. 
If we restrict our attention to the maps where the codomain is a surface then
the harmonic morphisms are horizontally weakly conformal maps with minimal fibers.
Good references for harmonic morphisms and there properties are \cite{Bai-Woo-book} and \cite{Gud-bib}.

In \cite{GudSack}, Gudmundsson and Sakovich develop a method for constructing complex-valued harmonic morphism on
the classical compact Lie groups by looking at eigenfamilies, simultaneous
eigenfunction of the Laplacian and the conformality operator, see Theorem \ref{HowToHarMor}.
We give some further properties of such eigenfamilies and extend the concept to maximal eigenfamilies.

In \cite{Gud-Sven-Ville}, the authors show how to construct these eigenfamilies using
representation theory. The Peter-Weyl theorem gives the eigenfunctions of the Laplacian
and decompositions into irreducible representation and Schur's lemma gives eigenfunctions
of the conformality operator. We show in this paper that the method
works for the exceptional Lie group $G_{2}$ as well.
We make use of the $7$-dimensional cross product to control the behavior of the
conformality operator. Our main result is the following.

\begin{theorem}\label{MainThm}
Let $\rho:G_{2}\to\Aut(\cn^{7})$ be the standard representation of the exceptional compact Lie group $G_{2}$.
For any $a,b\in\cn^{7}$ define
\[\phi_{ab}(g)=\SPE{\rho(g)a}{b},\]
where $\SPE{\cdot}{\cdot}$ is the complex-bilinear extension of the standard scalar product on $\rn^{7}$.
Suppose $p\in\cn^{7}$ satisfies $\SPE{p}{p}=0$. Then
\[\E_{p}=\{\phi_{ap}\,|\,a\in\cn^{7}\}\]
is an eigenfamily on $G_{2}$, i.e. there exist $\lambda,\mu\in\rn$ such that
\[\Delta(\phi)=\lambda\phi\quad\textrm{and}\quad\kappa(\phi,\psi)=\mu\phi\psi\]
for all $\phi,\psi\in\E_{p}$.
\end{theorem}

\section{Maximal eigenfamilies}

In this section we present the notion of an eigenfamily introduced in \cite{GudSack}. We give some
further properties of eigenfamilies and extend the notion to
maximal eigenfamilies, with the hope that this may help with future classifications.

Let $(M,g)$ be a Riemannian manifold.
For functions $\phi,\psi:(M,g)\to \cn$ we define the \textrm{Laplacian} $\Delta$ and the
\textrm{conformality} operator $\kappa$ by
\[\Delta(\phi)=\div(\grad(\phi))\]
and
\[\kappa(\phi,\psi)=g(\grad(\phi),\grad(\psi)),\]
here $g$ is the complex-bilinear extension of the metric $g$ and $\grad(\phi),\grad(\psi)$
are sections of the complexified tangent bundle $T^{\cn}M$.
It is well know that the Laplacian acts as
\[\Delta(\phi\psi)=\psi\Delta(\phi)+2\kappa(\phi,\psi)+\phi\Delta(\psi)\]
on a product of two functions.
If $\kappa(\phi,\phi)=0$ then $\phi$ is horizontally conformal and if $\Delta(\phi)=0$ then $\phi$ is a harmonic map and
if $\kappa(\phi,\phi)=\Delta(\phi)=0$ then $\phi$ is a harmonic morphism.

Given $\lambda\in\cn$ let $\F_{\lambda}$ be the eigenfunctions of the Laplacian with eigenvalue $\lambda$
\[\F_{\lambda}=\{\phi:(M,g)\to\cn\,|\,\Delta(\phi)=\lambda\phi\}\]
and given $\lambda,\mu\in\cn$ let $\F_{\lambda,\mu}$ be the simultaneous eigenfunctions of
the Laplacian and the conformality operator with eigenvalues $\lambda$ and $\mu$ respectively,
\[\F_{\lambda,\mu}=\{\phi\in \F_{\lambda}\,|\,\kappa(\phi,\phi)=\mu\phi\phi\}.\]
Even if $\phi,\psi\in\F_{\lambda,\mu}$ it may still happen that $\kappa(\phi,\psi)\neq \mu\phi\psi$.

\begin{definition}\cite{GudSack}
Let $(M,g)$ be a Riemannian manifold and $\lambda,\mu\in\cn$.
A subset $\E_{\lambda,\mu}$ of $\F_{\lambda,\mu}$ such that
$\kappa(\phi,\psi)=\mu\phi\psi$ for all $\phi,\psi\in\E_{\lambda,\mu}$ is said to
be an \textbf{eigenfamily}.
\end{definition}

In fact, a subset $\E\subseteq\F_{\lambda,\mu}$ is an eigenfamily if and only if, $\phi,\psi\in\E$ imply
$\phi+\psi\in\F_{\lambda,\mu}$, since
\[\kappa(\phi,\psi)=\frac{1}{2}(\kappa(\phi+\psi,\phi+\psi)-\kappa(\phi,\phi)-\kappa(\psi,\psi))=\mu\phi\psi.\]

$\F_{\lambda,\mu}$ is closed under composition by isometries, i.e. if $\phi\in\F_{\lambda,\mu}$
and $F$ is an isometry then $F^{*}(\phi)=\phi\circ F\in\F_{\lambda,\mu}$. 

\begin{proposition}
Let $(M,g)$ be a Riemannian manifold and $\E_{\lambda,\mu}$ be an eigenfamily on $M$.
Then $\sp_{\cn}\E_{\lambda,\mu}$ and $\bar{\E}_{\lambda,\mu}=\{\bar{\phi}|\phi\in\E_{\lambda,\mu}\}$
are also eigenfamilies.
\end{proposition}

\begin{proof}[\pro]
This follows from the fact that $\Delta$ is linear and $\kappa$ is bilinear.
\end{proof}

Given an eigenfamily it is easy to produce more eigenfamilies by taking the symmetric product,
as in the next definition and lemma.

\begin{definition}
Let $(M,g)$ be a Riemannian manifold and $\E_{\lambda,\mu}$ be an
eigenfamily on $M$. Define $\odot^{n}\E_{\lambda,\mu}$ by
\[\odot^{n}\E_{\lambda,\mu}=\sp_{\cn}\{\phi_{1}\cdots\phi_{n}\,|\,\phi_{i}\in\E_{\lambda,\mu}\textrm{ for }i=1\ldots n\}.\]
\end{definition}

\begin{lemma}\cite{GudSack}
Let $(M,g)$ be a Riemannian manifold and $\E_{\lambda,\mu}$ be an eigenfamily on $M$.
If $\phi\in\odot^{k}\E_{\lambda,\mu}$ and $\psi\in\odot^{l}\E_{\lambda,\mu}$ then
\[\kappa(\phi,\psi)=k l\mu\phi\psi\quad\textrm{and}\quad\Delta(\phi)=k(\lambda+(k-1)\mu)\phi.\]
In particular, $\odot^{n}\E_{\lambda,\mu}$ is an eigenfamily.
\end{lemma}

\begin{proof}[\pro]
The proof is by induction. The statement is obvious for $k=l=1$.
Assume $\kappa(\phi_{i},\psi_{j})=i j\mu\phi_{i}\psi_{j}$, for all $\phi_{i}\in \odot^{i}\E$ and
all $\psi_{j}\in\odot^{j}\E$ for $i\leq k$ and $j\leq l$.
We will show that the statement is true for all $i\leq k+1$ and all $j\leq l+1$.
Let $\theta,\Theta\in\E_{\lambda,\mu}$ then
\begin{align*}
\kappa(\phi_{k}\theta,\psi_{l})=&g(\theta\grad(\phi_{k})+\phi_{k}\grad(\theta),\grad(\psi_{l}))\\
=&\theta g(\grad(\phi_{k}),\grad(\psi_{l}))+\phi_{k}g(\grad(\theta),\grad(\psi_{l}))\\
=&(k l\mu+l\mu)\theta\phi_{k}\psi_{l}\\
=&\mu(k+1)l\theta\phi_{k}\psi_{1}.
\end{align*}
Similarly $\kappa(\phi_{k},\theta\psi_{l})=k(l+1)\mu\theta\phi_{k}\psi_{l}$,
$\kappa(\phi_{k}\Theta,\theta\psi_{l})=(k+1)(l+1)\mu\Theta\theta\phi_{k}\psi_{l}$.
For the Laplacian we have
\begin{align*}
\Delta(\phi_{k}\theta)=&\phi_{k}\Delta(\theta)
+2\kappa(\phi_{k},\theta)+\theta\Delta(\phi_{k})\\
=&\lambda\phi_{k}\theta+2\mu k \phi_{k}\theta+k(\lambda+(k-1)\mu)\phi_{k}\theta\\
=&(k+1)(\lambda+k\mu)\phi_{k}\theta.\qedhere
\end{align*}
\end{proof}

The following result from \cite{GudSack} details how eigenfamilies are used to produce harmonic morphisms.

\begin{theorem}\label{HowToHarMor}
Let $(M,g)$ be a Riemannian manifold, $\E_{\lambda,\mu}$ an eigenfamily on $M$ and
$\{\phi_{1},\ldots,\phi_{n}\}$ a basis for $\sp_{\cn}\E_{\lambda,\mu}$. Then
\[\frac{P(\phi_{1},\ldots,\phi_{n})}{Q(\phi_{1},\ldots,\phi_{n})}\]
is a harmonic morphism on $M\backslash\{p\in M\,|\,Q(\phi_{1},\ldots,\phi_{n})(p)=0\}$
for any linearly independent homogeneous polynomials $P,Q:\cn^{n}\to\cn$ of the
same degree $d>0$.
\end{theorem}
In order to obtain non-constant harmonic morphisms we need eigenfamilies with
at least two dimensional spans.

Now we will state some further properties of eigenfamilies on compact manifolds, we will start
with the possible eigenvalues.

\begin{proposition}
Let $(M,g)$ be a compact Riemannian manifold. If $\F_{\lambda,\mu}\neq \{0\}$ then $\lambda,\mu\in\rn$
and $\lambda,\mu\leq 0$.
\end{proposition}

\begin{proof}[\pro]
With the sign convention, for the Laplacian, that we have chosen it is well known that for compact
manifolds any eigenvalue $\lambda$ has to be real and $\lambda>0$ imply $\F_{\lambda}=\{0\}$.

Let $\phi\in \F_{\lambda,\mu}$ be a non-zero function. Then $\phi^{2}\in \F_{2(\lambda+\mu),4\mu}$ is non-zero.
If $\mu$ is not real then $\phi^{2}$ will have a complex eigenvalue of the Laplacian, which
is a contradiction.
If $\mu>0$ then there exist a $k_{0}\in\zn^{+}$ such that $\lambda +(2 k_{0}-1)\mu>0$. Then
\[\Delta(\phi^{2k_{0}})=2k_{0}(\lambda+(2k_{0}-1)\mu)\phi^{2k_{0}}\]
so $\phi^{2k_{0}}\in\F_{2k_{0}(\lambda +(2 k_{0}-1)\mu)}$ but
$\F_{2k_{0}(\lambda +(2 k_{0}-1)\mu)}=\{0\}$ since $2k_{0}(\lambda +(2 k_{0}-1)\mu)>0$.
This is a contradiction.
\end{proof}

The following result shows the importance of letting the functions be complex valued.

\begin{proposition}
Let $(M,g)$ be a compact Riemannian manifold.
If a function $\phi\in \F_{\lambda,\mu}$ is real-valued then it is constant.
\end{proposition}

\begin{proof}[\pro]
Since $\phi$ is real-valued $\kappa(\phi,\phi)=|\nabla(\phi)|^{2}=\mu|\phi|^{2}$, so $\mu\geq 0$.
But we know that $\mu\leq 0$.
So $\mu=0$ which imply $|\nabla(\phi)|^{2}=0$ and $\phi$ is constant.
\end{proof}

Any subset of an eigenfamily is again an eigenfamily, so to classify them we want to find
maximal subspaces of $\F_{\lambda}$ that are eigenfamilies.

\begin{definition}
Suppose $\E_{\lambda,\mu}$ is an eigenfamily and that for all
$\theta\in \F_{\lambda,\mu}\backslash\E_{\lambda,\mu}$ there exists a
$\phi\in\E_{\lambda,\mu}$ such that $\kappa(\theta,\phi)\neq \mu\theta\phi$. Then
we call $\E_{\lambda,\mu}$ a \textbf{maximal} eigenfamily.
\end{definition}

Since the span of an eigenfamily is also an eigenfamily, maximal eigenfamilies are vector spaces.
In fact they are the largest possible linear subsets of $\F_{\mu,\lambda}$.

\begin{proposition}
Let $(M,g)$ be a compact Riemannian manifold and $\lambda<0$ and $\E_{\lambda,\mu}$ an
eigenfamily. Then $\E_{\lambda,\mu}\cap\bar{\E}_{\lambda,\mu}=\{0\}$.
\end{proposition}

\begin{proof}[\pro]
Let $\phi\in\E_{\lambda,\mu}\cap\bar{\E}_{\lambda,\mu}$. Then $\phi,\bar{\phi}\in\E_{\lambda,\mu}$
so  $|\phi|^{2}=\phi\bar{\phi}\in \F_{2(\lambda+\mu),4\mu}$ but since it is
real-valued it must be constant. Then $\lambda<0$ imply $\phi=0$.
\end{proof}

We see that the dimension of a maximal eigenfamily $\E_{\lambda,\mu}$ must be less than or equal to half
the dimension of $\F_{\lambda}$.

\section{Constructing eigenfamilies on compact Lie groups}

In this section we describe the method to obtain eigenfamilies using unitary representations
of Lie groups given in \cite{Gud-Sven-Ville}. Let $G$ be a compact Lie group with a
bi-invariant metric. By the Theorem of Highest Weight, see Theorem 5.110 in \cite{Knapp},
the irreducible representations of $G$ are in one-to-one with the dominant algebraically
integral roots, we denote this set of roots by $\Gamma$. The correspondence is such that each $\gamma\in\Gamma$
is the highest weight for the representation $\rho_{\gamma}:G\to\Aut(V_{\gamma})$.
By the Peter-Weyl theorem, see Theorem 4.20 in \cite{Knapp}, we have
\[L^{2}(G)=\bigoplus_{\gamma\in\Gamma}M(V_{\gamma}),\]
where $M(V_{\gamma})$ is the set of functions spanned by the elements of the
matrix $\rho_{\gamma}(g):V_{\gamma}\to V_{\gamma}$. From now on we suppress $\gamma$ and $\rho$
denotes any irreducible representation of $G$.

Given $a,b\in V^{\cn}$ define $\phi_{ab}:G\to\cn$ by $\phi_{ab}(g)=\left<\rho(g)a,b\right>$.
Then $\phi_{ab}\in M(V)$, in fact these function span all of $M(V)$.
Let $\sigma:\g\to\End(V)$ denote the Lie algebra representation related to $\rho$.
Then we have $X_{g}(\phi_{ab})=\SPE{\rho(g)\sigma(X)a}{b}$ for any left-invariant
vector field $X\in\g$.

Let $\{X_{i}\}_{i=1}^{n}$ be an orthonormal basis for the Lie algebra $\g$ of $G$. Then
$\{-X_{i}\}_{i=1}^{n}$ is its dual basis with respect to the Killing form.
The Laplacian acts on $\phi_{ab}$ as
\begin{align*}
\Delta(\phi_{ab})(g)&=\sum_{i}X_{i}(X_{i}(\phi_{ab}))(g)\\
&=\sum_{i}\SPE{\rho(g)\sigma(X_{i})^{2}a}{b}\\
&=\SPE{\rho(g)(\sum_{i}\sigma(X_{i})^{2})a}{b}\\
&=-\SPE{\rho(g)(\sum_{i}\sigma(X_{i})\sigma(-X_{i}))a}{b}.
\end{align*}
Now $\sum_{i}\sigma(X_{i})\sigma(-X_{i})$ is the Casimir operator of the representation $\sigma$.
Since the representation is irreducible it is a scalar times the identity. The scalar is in fact,
see Theorem 5.28 in \cite{Knapp},
$\SPE{\gamma}{\gamma+2\delta}$, where $\gamma$ is the highest weight and $\delta$ is
half the sum of the positive roots. Thus
\[\Delta(\phi_{ab})=\lambda\phi_{ab},\]
where $\lambda=-\SPE{\gamma}{\gamma+2\delta}$.
In the notation of the previous section $\F_{\lambda}=M(V_{\gamma})$.

Similarly $\kappa$ is given by
\begin{align*}
\kappa(\phi_{ab},\phi_{cd})(g)&=\sum_{i=1}^{n}(X_{i})_{g}(\phi_{ab})(X_{i})_{g}(\phi_{cd})\\
&=\sum_{i=1}^{n}\SPE{\rho(g)\sigma(X_{i})a}{b}\SPE{\rho(g)\sigma(X_{i})c}{d}.
\end{align*}
Define $Q$ to be the value of $\kappa(\phi_{ab},\phi_{cd})$ in the identity element $e\in G$, i.e.
\[Q(a,b,c,d)=\kappa(\phi_{ab},\phi_{cd})(e)=\sum_{i}\left<\sigma(X_{i})a,b\right>\left<\sigma(X_{i})c,d\right>.\]

If we assume that the representation $\rho$ is of real type then
\begin{align*}
\kappa(\phi_{ab},\phi_{cd})(g)=&\sum_{i}\left<\rho(g)\sigma(X_{i})a,b\right>\left<\rho(g)\sigma(X_{i})c,d\right>\\
=&\sum_{i}\left<\rho(g)\sigma(X_{i})\rho(g)^{-1}\rho(g)a,b\right>\left<\rho(g)\sigma(X_{i})\rho(g)^{-1}\rho(g)c,d\right>\\
=&\sum_{i}\left<\sigma(Y_{i})\rho(g)a,b\right>\left<\sigma(Y_{i})\rho(g)c,d\right>\\
=&Q(\rho(g)a,b,\rho(g)c,d),
\end{align*}
since $\{Y_{i}=gX_{i}g^{-1}\}$ is an other orthonormal basis for $\g$.

$Q$ is skew-symmetric in $a,b$ and in $c,d$ and
symmetric in the pairs $(a,b),(c,d)$, thus we can consider $Q$ as a symmetric map
$Q:\bigwedge^{2} V_{\lambda}\times \bigwedge^{2}V_{\lambda}\to \rn$.
By extending the scalar product to be complex bilinear we can see $Q$ as a symmetric map
$Q:\bigwedge^{2}V_{\lambda}^{\cn}\times\bigwedge^{2}V_{\lambda}^{\cn}\to\cn$.

\begin{theorem}\cite{Gud-Sven-Ville}\label{HowToEigenFam}
Let $G$ be a compact Lie group and $\rho:G\to \Aut(V)$ be a unitary representation of real type.
If there exists a $\mu\in\rn$ such that $Q(a\wedge p,b\wedge p)=\mu\SPW{a\wedge p}{b\wedge p}$
for all $a,b,p\in V^{\cn}$, then $\E(q)=\{\phi_{aq}(g)=\SPE{\rho(g)a}{q}\,|\,a\in V^{\cn}\}$ is an eigenfamily
for any $q\in V^{\cn}$ such that $\SPE{q}{q}=0$.
\end{theorem}

\begin{proof}[\pro]
We already know that $\Delta(\phi_{ab})=\lambda\phi_{ab}$ by the Peter-Weyl theorem.
For any $g\in G$
\begin{align*}
\kappa(\phi_{aq},\phi_{bq})(g)=&Q(\rho(g)a\wedge q,\rho(g)b\wedge q)\\
=&\mu\SPW{\rho(g)a\wedge q}{\rho(g)b\wedge q}\\
=&\mu(\SPE{\rho(g)a}{\rho(g)b}\SPE{q}{q}-\SPE{\rho(g)a}{q}\SPE{\rho(g)b}{q})\\
=&-\mu\phi_{aq}(g)\phi_{bq}(g).\qedhere
\end{align*}
\end{proof}

For a Euclidean vector space $V$ define the isomorphism $R:\bigwedge^{2}V \to \so(V)$ by
\[R(a\wedge b)(v)=\SPE{a}{v}b-\SPE{b}{v}a.\]
Define the exterior square representation
$\tilde{\rho}:G\to\Aut(\bigwedge^{2}V)$ by
\[\tilde{\rho}(g)(a\wedge b)=(\rho(g)a)\wedge (\rho(g)b)\]
and the representation $\hat{\rho}:G\to\Aut(\so(V))$ by
\[\hat{\rho}(g)(A)=\rho(g)A\rho(g)^{-1}.\]
These representations are equivalent since
\begin{align*}
(\hat{\rho}(g)R(a\wedge b))(v)=&\rho(g)R(a\wedge b)(\rho(g)^{-1}v)\\
=&\rho(g)(\SPE{a}{\rho(g)^{-1}v}b-\SPE{b}{\rho(g)^{-1}v}a)\\
=&\SPE{\rho(g)a}{v}\rho(g)b-\SPE{\rho(g)b}{v}\rho(g)a\\
=&R(\tilde{\rho}(g)(a\wedge b))(v).
\end{align*}

The scalar product on $\bigwedge^{2}V$ is given by
\[\SPW{a\wedge b}{c\wedge d}=\SPE{a}{c}\SPE{b}{d}-\SPE{a}{d}\SPE{b}{c}\]
and the scalar product on $\so(V)$ by
\[\SPS{A}{B}=\sum_{i}\SPE{A(e_{i})}{B(e_{i})}=\tr(A^{t}B).\]

\begin{lemma}
Let $V$ be a Euclidean space, $A\in\so(V)$ and $a,b,c,d\in V$. Then
\begin{itemize}
\item[(i)] $\SPS{A}{R(a\wedge b)}=2\SPE{A(a)}{b}$, and
\item[(ii)] $\SPS{R(a\wedge b)}{R(c\wedge d)}=2\SPW{a\wedge b}{c\wedge d}$.
\end{itemize}
\end{lemma}

\begin{proof}[\pro]
By the definition of the scalar product
\begin{align*}
\SPS{A}{R(a\wedge b)}&=\sum_{i}\SPE{A(e_{i})}{\SPE{a}{e_{i}}b-\SPE{b}{e_{i}}a}\\
&=\sum_{i}(\SPE{A(\SPE{a}{e_{i}}e_{i})}{b}-\SPE{A(\SPE{b}{e_{i}}e_{i})}{a})\\
&=2\SPE{A(a)}{b}.
\end{align*}
Set $A=R(c\wedge d)$, then
\begin{align*}
\SPS{R(a\wedge b)}{R(c\wedge d)}&=2\SPE{R(a\wedge b)(c)}{d}\\
=&2\SPE{\SPE{a}{c}b-\SPE{b}{c}a}{d}\\
=&2(\SPE{a}{c}\SPE{b}{d}-\SPE{b}{c}\SPE{a}{d})\\
=&2\SPW{a\wedge b}{c\wedge d}.\qedhere
\end{align*}
\end{proof}

The next result tells us that $Q$ acts as a projection onto the Lie algebra of $G$.

\begin{theorem}\cite{Gud-Sven-Ville}
Let $G$ be a Lie group and $\rho:G\to\Aut(V)$ be a representation of real type.
Then for all $a,b,c,d\in V$
\[Q(a\wedge b,c\wedge d)=\SPW{P_{\g}(a\wedge b)}{c\wedge d}\]
where $P_{\g}$ is the projection onto the Lie algebra $\g$ of $G$.
\end{theorem}

\begin{proof}[\pro]
Let $a,b,c,d\in V$ and $\{X_{i}\}_{i}^{n}$ be an orthonormal basis for $\g$. The
scalar product on $\g$ is such that the Lie algebra representation $\sigma:\g\to\so(V)$
is an isometry onto its image, meaning $\SPS{\sigma(X_{i})}{\sigma(X_{j})}=\delta_{ij}$.
\begin{align*}
Q(a\wedge b,c\wedge d)&=\sum_{i=1}^{n}\SPE{\sigma(X_{i}) a}{b}\SPE{\sigma(X_{i}) c}{d}\\
&=\frac{1}{4}\sum_{i=1}^{n}\SPS{\sigma(X_{i})}{R(a\wedge b)}\SPS{\sigma(X_{i})}{R(c\wedge d)}\\
&=\frac{1}{4}\SPS{\sum_{i=1}^{n}\SPS{\sigma(X_{i})}{R(a\wedge b)}\sigma(X_{i})}{R(c\wedge d)}\\
&=\frac{1}{4}\SPS{P_{\g}(R(a\wedge b))}{R(c\wedge d)}\\
&=\frac{1}{2}\SPW{P_{\g}(a\wedge b)}{c\wedge d}.\qedhere
\end{align*}
\end{proof}

Everything will still hold if we take $a,b,c,d\in V^{\cn}$ and use the complex bilinear extension
of the scalar product.

\section{Eigenfamilies on $G_{2}$}

In this section we show that for $G=G_{2}$ the conditions of Theorem \ref{HowToEigenFam} are satisfied, hence
there exist eigenfamilies on $G_{2}$. Let $\rho:G_{2}\to\Aut(V)$ denote the standard representation of $G_{2}$.
The exterior square representation
$\tilde{\rho}:G_{2}\to\Aut(\bigwedge^{2}V)$ has two irreducible subspaces
$\bigwedge^{2}V=\g_{2}\oplus W$, see page 353 in \cite{Fulton-Harris}. As $\dim(\g_{2})=14$ we have $\dim(W)=7$.

The following result is standard, see page 408 in \cite{Calabi}.
\begin{lemma}\label{CrossProd} Let $V$ be a seven dimensional Euclidean space and
$\times:V\times V\to V$ be the seven dimensional cross-product. Then
\[\SPE{v\times u}{w}=-\SPE{u}{v\times w}\]
and
\[u\times(v\times w)+v\times(u\times w)=\SPE{u}{w}v+\SPE{v}{w}u-2\SPE{u}{v}w\]
and for all $u,v,w\in V$.
\end{lemma}

In fact $G_{2}$ is exactly the set $\{g\in SO(7)\,|\,g(v\times w)=(gv)\times(gw)\}$.
Next we present 3 lemmas that describe the interaction between the seven dimensional cross product, 
the wedge product and the projection onto $W$.

\begin{lemma}
Let $V$ be a seven dimensional Euclidean space and $\times:V\times V\to V$ be the
seven dimensional cross product. Then
\[\SPE{a\times u}{b\times v}+\SPE{a\times v}{b\times u}=\SPW{a\wedge u}{b\wedge v}+\SPW{a\wedge v}{b\wedge u}\]
for all $a,b,u,v\in V$.
\end{lemma}

\begin{proof}[\pro]
Using Lemma \ref{CrossProd} we find that
\begin{align*}
\SPE{a\times u}{b\times v}+\SPE{a\times v}{b\times u}&=\SPE{a}{u\times(b\times v)}
+\SPE{a}{v\times (b\times u)}\\
&=\SPE{a}{u\times(b\times v)+v\times (b\times u)}\\
&=-\SPE{a}{u\times(v\times b)+v\times (u\times b)}\\
&=-\SPE{a}{\SPE{u}{b}v+\SPE{b}{v}u-2\SPE{u}{v}b}\\
&=\SPW{a\wedge u}{b\wedge v}+\SPW{a\wedge v}{b\wedge u}.\qedhere
\end{align*}
\end{proof}

\begin{lemma}
Let $V$ be a seven dimensional Euclidean space and
$\times:V\times V\to V$ be the seven dimensional cross-product.
Define the operator $L:V\to \so(V)$, $v\mapsto L_{v}$, where
\[L_{v}(w)=v\times w\]
for $v,w\in V$.
Then $L(V)=R(W)$ and $\SPS{L_{v}}{L_{w}}=6\SPE{v}{w}$.
\end{lemma}

\begin{proof}[\pro]
For the first part, $L_{v}=0$ means that $v\times w=0$ for all $w$ which implies $v=0$, so $L$
is injective and therefore $L(V)$ is seven dimensional. Now
\begin{align*}
(\hat{\rho}(g)L_{v})(w)&=(\rho(g)L_{v}\rho(g)^{-1})(w)\\
&=\rho(g)(v\times(\rho(g)^{-1}w))\\
&=(\rho(g)v)\times w\\
&=L_{\rho(g)v}(w).
\end{align*}
This means that $\hat{\rho}(g)L_{v}\in L(V)$ for all $g\in G_{2}$ and $v\in V$ thus $L(V)$ is a invariant
$7$-dimensional subspace of $\so(V)$. Since $\g_{2}$ is an irreducible $14$-dimensional invariant
subspace we must have $L(V)=R(W)$.

For the second part, let $\{e_{i}\}_{i=1}^{7}$ be an orthonormal basis for $V$. Then
\begin{align*}
\SPS{L_{v}}{L_{w}}&=\sum_{i}\SPE{L_{v}(e_{i})}{L_{w}(e_{i})}\\
&=\sum_{i}\SPE{v\times e_{i}}{w\times e_{i}}\\
&=\sum_{i}\SPE{e_{i}\times v}{e_{i}\times w}\\
&=-\sum_{i}\SPE{v}{e_{i}\times (e_{i}\times w)}\\
&=-\sum_{i}\SPE{v}{-\SPE{e_{i}}{e_{i}}w+\SPE{e_{i}}{w}e_{i}}\\
&=\sum_{i}(\SPE{v}{w}\SPE{e_{i}}{e_{i}}-\SPE{v}{e_{i}}\SPE{w}{e_{i}})\\
&=7\SPE{v}{w}-\SPE{v}{w}\\
&=6\SPE{v}{w}.\qedhere
\end{align*}
\end{proof}

\begin{lemma}
Let $V$ be a seven dimensional Euclidean space and
$\times:V\times V\to V$ be the seven dimensional cross-product.
Let $a,b\in V$ then
\[P_{R(W)}(R(a\wedge b))=\frac{1}{3}L_{a\times b}.\]
\end{lemma}

\begin{proof}[\pro]
Let $\{e_{k}\}$ be an orthonormal basis for $V$. Then $\{\frac{1}{\sqrt{6}}L_{e_{k}}\}$ is
an orthonormal basis for $R(W)$ so
\begin{align*}
P_{R(W)}(R(a\wedge b))&=\frac{1}{6}\sum_{k}\SPS{R(a\wedge b)}{L_{e_{k}}}L_{e_{k}}\\
&=\frac{1}{6}\sum_{i,k}\SPE{R(a\wedge b)(e_{i})}{e_{k}\times e_{i}}L_{e_{k}}\\
&=-\frac{1}{6}\sum_{i,k}\SPE{R(a\wedge b)(e_{i})}{e_{i}\times e_{k}}L_{e_{k}}\\
&=\frac{1}{6}\sum_{i,k}\SPE{e_{i}\times R(a \wedge b)(e_{i})}{e_{k}}L_{e_{k}}\\
&=\frac{1}{6}\sum_{i,k}\SPE{e_{i}\times(\SPE{a}{e_{i}}b-\SPE{b}{e_{i}}a)}{e_{k}}L_{e_{k}}\\
&=\frac{1}{6}\sum_{k}\SPE{\sum_{i}(\SPE{a}{e_{i}}e_{i}\times b-\SPE{b}{e_{i}}e_{i}\times a)}{e_{k}}L_{e_{k}}\\
&=\frac{1}{6}\sum_{k}\SPE{2a\times b}{e_{k}}L_{e_{k}}\\
&=\frac{1}{3}L_{a\times b}.\qedhere
\end{align*}
\end{proof}

Thus for all $a,b,c,d\in V$
\begin{align*}
Q(a\wedge b,c\wedge d)&=\frac{1}{4}\SPS{P_{\g_{2}}(R(a\wedge b))}{R(c\wedge d)}\\
&=\frac{1}{4}\SPS{R(a\wedge b)-P_{R(W)}(R(a\wedge b))}{R(c\wedge d)}\\
&=\frac{1}{4}\left(\SPS{R(a\wedge b)}{R(c\wedge d)}-\SPS{P_{R(W)}(R(a\wedge b))}{P_{R(W)}(R(c\wedge d))}\right)\\
&=\frac{1}{2}\SPW{a\wedge b}{c\wedge b}-\frac{1}{4}\SPS{\frac{1}{3}L_{a\times b}}{\frac{1}{3}L_{c\times d}}\\
&=\frac{1}{2}\SPW{a\wedge b}{c\wedge b}-\frac{1}{6}\SPE{a\times b}{c\times d}.
\end{align*}

The next result details how $Q$ works on the set $\{a\wedge p\,|\,a\in V^{\cn}\}$ where
$p\in V^{\cn}$ is given.

\begin{theorem}\label{EigenOnG2}
Let $\rho:G_{2}\to \Aut(V)$ be the standard representation of $G_{2}$ and $\sigma:\g_{2}\to\End(V)$ be
the corresponding representation of its Lie algebra $\g_{2}$. Define $Q:\bigwedge^{2}V^{\cn}\to\bigwedge^{2}V^{\cn}$ by
\[Q(a\wedge b,c\wedge d)=\sum_{i=1}^{14}\SPE{\sigma(X_{i})a}{b}\SPE{\sigma(X_{i})c}{d}\]
for $a,b,c,d\in V^{\cn}$, where $\{X_{i}\}_{i=1}^{14}$ is an orthonormal basis for $\g_{2}$. Then
\[Q(a\wedge p,b\wedge p)=\frac{1}{3}\SPW{a\wedge p}{b\wedge p}\]
for all $a,b,p\in V^{\cn}$.
\end{theorem}

\begin{proof}[\pro]
Let $p=u+i v\in V^{\cn}$ and $a,b\in V$. Then
\begin{align*}
Q(a\wedge(u+i v),b\wedge(u+i v))=&Q(a\wedge u,b\wedge u)
-Q(a\wedge v,b\wedge v)\\
&+i(Q(a\wedge u,b\wedge v)+Q(a\wedge v,b\wedge u))\\
=&(\frac{1}{2}\SPW{a\wedge u}{b\wedge u}-\frac{1}{6}\SPE{a\times u}{b\times u})\\
&-(\frac{1}{2}\SPW{a\wedge v}{b\wedge v}-\frac{1}{6}\SPE{a\times v}{b\times v})\\
&+i(\frac{1}{2}\SPW{a\wedge u}{b\wedge v}-\frac{1}{6}\SPE{a\times u}{b\times v})\\
&+i(\frac{1}{2}\SPW{a\wedge v}{b\wedge u}-\frac{1}{6}\SPE{a\times v}{b\times u})\\
=&(\frac{1}{2}\SPW{a\wedge u}{b\wedge u}-\frac{1}{6}\SPE{a\times u}{b\times u})\\
&-(\frac{1}{2}\SPW{a\wedge v}{b\wedge v}-\frac{1}{6}\SPE{a\times v}{b\times v})\\
&+i(\frac{1}{2}\SPW{a\wedge u}{b\wedge v}+\frac{1}{2}\SPW{a\wedge v}{b\wedge u})\\
&-i(\frac{1}{6}\SPE{a\times u}{b\times v}+\frac{1}{6}\SPE{a\times v}{b\times u})\\
=&\frac{1}{3}\SPW{a\wedge u}{b\wedge u}-\frac{1}{3}\SPW{a\wedge v}{b\wedge v}\\
&+i(\frac{1}{3}\SPW{a\wedge u}{b\wedge v}+\frac{1}{3}\SPW{a\wedge v}{b\wedge u})\\
=&\frac{1}{3}\SPW{a\wedge (u+i v)}{b\wedge (u+i v)}.
\end{align*}
Since the expression is complex linear in $a,b$ the formula holds for $a,b\in V^{\cn}$ as well.
\end{proof}

We get our main result Theorem \ref{MainThm} by combining
Theorem \ref{HowToEigenFam} and Theorem \ref{EigenOnG2}. 

\section{Acknowledgments}
Thanks to Martin Svensson for suggesting the seven-dimensional cross-product.


\end{document}